\documentclass[reqno]{amsart}
\usepackage{amsfonts,amsmath,amsthm,amssymb,mathrsfs}
\usepackage{color}
\usepackage{amsmath,amssymb,amsfonts,bm}
\usepackage{graphicx}
\usepackage{CJKutf8}
\usepackage{newunicodechar}
\usepackage{xcolor}
\numberwithin{equation}{section}
\usepackage[pdfstartview=FitH,colorlinks=true]{hyperref}
\hypersetup{urlcolor=red, linkcolor=blue,citecolor=red, CJKbookmarks=true}
\allowdisplaybreaks

\newtheorem{theorem}{Theorem}[section]

\newtheorem{lemma}[theorem]{Lemma}

\theoremstyle{definition}

\newtheorem{remark}{Remark}[section]

\newcommand{\norm}[1]{\left\Vert#1\right\Vert}

\makeatletter

\newcommand{\Rmnum}[1]{\expandafter\@slowromancap\romannumeral #1@}
\makeatother

\DeclareMathOperator\divg{div}

\title[On the radius of spatial analyticity for the inviscid Boussinesq equations]
{On the radius of spatial analyticity for the inviscid Boussinesq equations}
\author{Feng Cheng, Chao-Jiang Xu}
\date{}

\address{\noindent \textsc{Feng Cheng, Hubei Key Laboratory of Applied Mathematics, Faculty of Mathematics and Statistics, Hubei university, Wuhan 430062, P.R. China}}
\email{fengcheng@hubu.edu.cn}

\address{\noindent \textsc{Chao-Jiang Xu,
Department of Mathematics, Nanjing University of Aeronautics and Astronautics, 211106 Nanjing, China\\
and\\
Universit\'e de Rouen, CNRS UMR 6085, Laboratoire de Math\'ematiques, 76801 Saint-Etienne du Rouvray, France}}
\email{xuchaojiang@nuaa.edu.cn}
\begin{document}


\begin{abstract}
In this paper, we study the problem of analyticity of smooth solutions of the inviscid Boussinesq equations. If the initial datum is real-analytic, the solution remains real-analytic on the existence interval. By an inductive method we can obtain lower bounds on the radius of spatial analyticity of the smooth solution.
\end{abstract}

\maketitle

\section{Introduction}
In this paper, we consider the following multi-dimensional invisicd Boussinesq equation on the torus $\mathbb{T}^d$,
\begin{equation}\label{1.1}
\left\{
\begin{aligned}
&\partial_t{ u}+({ u}\cdot\nabla){ u}+\nabla p=\theta { e_d},\\
&\partial_t\theta+({ u}\cdot\nabla)\theta=0,\\
&\divg { u}=0,\\
&{ u}(x,0)={ u}_0(x), \theta(x,0)=\theta_0(x),
\end{aligned}
\right.
\end{equation}
with $\divg { u}_0=0$. Here, ${ u}=(u_1,\ldots,u_d)$ is the velocity field, $p$ the scalar pressure, and $\theta$ the scalar density. ${ e_d}$ denotes the vertical unit vector $(0,\ldots,0,1)$.
The Boussinesq systems play an important role in geophysical fluids such as atmospheric fronts and oceanic circulation (see, e.g., \cite{G,M,P}). Moreover, the Boussinesq systems are important for the study of the Rayleigh-Benard convection, see \cite{CD,DG}.

Besides the physical importance, the invisicd Boussinesq equations can also be viewed as simplified model compared with the Euler equation. In the case $d=2$, the 2D inviscid Boussinesq equations share some key features with the 3D Euler equations such as the vortex stretching mechanism. It was also pointed out in \cite{MB} that the 2D invisicid Boussinesq equations are identical to the Euler equations for the 3D axisymmetric swirling flows outside the symmetric axis.

The inviscid Boussinesq equations have been studied by many authors through the years, for instance \cite{Chae-K,Chae-N, ES, G, LWZ, Tan, Xu, Y}. Specially, Chae and Nam \cite{Chae-N} studied local existence and uniqueness of the inviscid Boussinesq equation and some blow-up criterion in the Sobolev space, Yuan \cite{Y} and Liu et al. \cite{LWZ} in the Besov space, Chae and Kim \cite{Chae-K} and Cui et al. \cite{Cui} in the H{\"o}lder spaces, Xiang and Yan \cite{Xiang-Yan} in the Triebel-Lizorkin-Lorentz spaces. It was remarked that the global regularity for the inviscid Boussinesq equations even in two dimensions is a challenging open problem in mathematical fluid mechanics.

In this paper, we are concerned with the analyticity of smooth solutions of the inviscid Boussinesq equations \eqref{1.1}. The analyticity of the solution for Euler equations in the space variables, for analytic initial data is an important issue, studied in \cite{AM,B,BB,BBZ,KV1,KV2,LO}. In particular, Kukavica and Vicol \cite{KV1} studied the analyticity of solutions for the Euler equations and obtained that the radius of analyticity $\tau(t)$ of any analytic solution $u(t,x)$ has a lower bound
\begin{equation}\label{1.2}
\tau(t)\geq C(1+t)^{-2}\exp\left(-C_0\int_0^t \|\nabla u(s,\cdot)\|_{L^\infty}ds \right)
\end{equation}
for a constant $C_0>0$ depending on the dimension and $C>0$ depending on the norm of the initial datum in some finite order Sobolev space. The same authors in \cite{KV2} obtained a better lower bound for $\tau(t)$ for the Euler equations in a half space replacing $(1+t)^{-2}$ by $(1+t)^{-1}$ in \eqref{1.2}. In \cite{Cheng-X}, we have investigated the Gevrey analyticity of the smooth solution for the ideal MHD equations following the method of \cite{KV1}. The approach used in \cite{KV1,KV2,LO,Cheng-X} relies on the energy method in infinite order Gevrey-Sobolev spaces. Recently, Cappiello and Nicola \cite{CN} developed a new inductive method to simplify the proof of \cite{KV1,KV2}. In this paper, we shall apply this inductive method to study the analyticity of smooth solution for the inviscid Boussinesq equations. The main additional difficulty arise from the estimate of the weak coupling term $u\cdot\nabla\theta$.

The paper is organized as follows. In Section \ref{Section 2}, we will give some notations and state our main results. In Section \ref{Section 3}, we first recall some known results and then give some lemmas which are needed to prove the main Theorem. In Section \ref{Section 4}, we finish the proof of Theorem \ref{Theorem 2.1}.

\section{Notations and Main Theorem}\label{Section 2}

In this section we will give some notations and function spaces which will be used throughout the following arguments. Throughout the paper, $C$ denotes a generic constant which may vary from line to line. Since we work on the torus $\mathbb{T}^d$ throughout the paper, we shall write the function space $L^2$ or $H^k$ to represnt the functions that are squre integrable or squre integrable up to $k$-th derivative without mentioning the domain $\mathbb{T}^d$. 

Let $v=(v_1,\ldots,v_d)$ be a vector function, we say that $v\in L^2$ which means $v_i\in L^2$ for each $1\leq i\leq d$. We denote the $L^2$ norm of $v$ by $\|v\|_{L^2}=\sqrt{\sum_{1\leq i\leq d} \|v_i\|_{L^2}^2}$. Let $\rho$ be a scalar function, we say the pare $(v,\rho)\in L^2$ if $v,\rho\in L^2$. We denote the $L^2$ norm of the pare $(v,\rho)$ to be
\begin{equation}\nonumber
\|(v,\rho)\|_{L^2}=\sqrt{\|v\|_{L^2}^2+\|\rho\|_{L^2}^2}.
\end{equation}
Denote $\langle\cdot,\cdot\rangle$ to be the inner product in $L^2$ either for vector function or scalar function.

In \cite{LM}, it is stated that a smooth function $f$ is uniformly analytic in $\mathbb{T}^d$ if there exist $M, \tau>0$ such that
\begin{equation}\label{2.1}
\|\partial^\alpha f\|_{L^\infty}\leq M\frac{|\alpha|!}{\tau^{|\alpha|}},
\end{equation}
for all multi-indices $\alpha=(\alpha_1,\ldots,\alpha_d)\in \mathbb{N}_0^d$, where $|\alpha|=\alpha_1+\ldots+\alpha_d$. The supremum of the constant $\tau>0$ in \eqref{2.1} is called the radius of the analyticity of $f$. Notice that we can also replace the $L^\infty$ norm with a Sobolev norm $H^k, k\geq0$.

Let $n\geq0$ be an integer and $\alpha,\beta\in \mathbb{N}^d$ be multi-indices, then the sequence
\begin{equation}\nonumber
M_n=\frac{n!}{(n+1)^2}
\end{equation}
satisfies
\begin{equation}\label{2.1+}
\sum_{\beta<\alpha}{\alpha\choose\beta}M_{|\alpha-\beta|}M_{|\beta|+1}\leq C|\alpha|M_{|\alpha|},
\end{equation}
for all multi-index $\alpha,\beta\in\mathbb{N}^d$ and some universal constant $C$, for proof please refer to \cite{AM}.

With these notations, we can state our main results.

\begin{theorem}\label{Theorem 2.1}
Let $k>\frac{d}{2}+1$, $(u_0,\theta_0)$ be analytic in $\mathbb{T}^d$, satisfying $\divg u_0=0$ and 
\begin{equation}\label{2.2}
\|\partial^\alpha (u_0,\theta_0) \|_{H^k}\leq BA^{|\alpha|-1}|\alpha|!/(|\alpha|+1)^2, \quad \alpha\in\mathbb{N}^d,
\end{equation}
for some $B\geq \frac{9}{4}\|(u_0,\theta_0)\|_{H^{2k+1}}$ and $A\geq1$. Let $\big(u(t,x),\theta(t,x)\big)$ be the corresponding $H^k$ maximal solution of the inviscid Boussinesq equations \eqref{1.1}, with the initial datum $(u_0,\theta_0)$. Then $\big(u(t,x),\theta(t,x)\big)$ is analytic, and there exists constants $C_0, C_1>0$, depending only on $k$ and $d$, such that the radius of analyticity satisfying
\begin{equation}\label{2.3}
\tau(t)\geq \frac{1}{A(1+C_1Bt)}\exp\left(-C_0\int_0^t \big(1+\|\nabla u(s,\cdot)\|_{L^\infty}+\|\nabla \theta(s,\cdot)\|_{L^\infty}\big)ds \right).
\end{equation}
\end{theorem}

\begin{remark}
In the case $\theta=0$, Theorem \ref{Theorem 2.1} recovers the result of Kukavica and Vicol \cite{KV1} and Cappiello and Nicola \cite{CN} for the incompressible Euler equation.
\end{remark}

\begin{remark}
When the dimension $d=2$, the blow-up criterion proved by Chae and Nam in \cite{Chae-N} stated that the solution remains smooth up to $T$ as long as $\int_0^T \|\nabla\theta(\cdot,s)\|_{L^\infty} ds<\infty$. So, it will be very interesting if the quantity $\int_0^t \big(1+\|\nabla u(\cdot,s)\|_{L^\infty}+\|\nabla\theta(\cdot,s)\|_{L^\infty} \big)ds$ in the lower bound of the radius of analytic solution in \eqref{2.3} can be replaced by $\int_0^t \|\nabla\theta(\cdot,s)\|_{L^\infty}ds$.
\end{remark}

\section{The estimate of the Sobolev norm}\label{Section 3}
In order to prove the main Theorem \ref{Theorem 2.1}, we recall the following results about the local existence and uniqueness of $H^k$-solution of the inviscid Boussinesq equations \eqref{1.1} which is a proposition in \cite{WX} for $d=3$.
\begin{theorem}[Wang-Xie, Proposition 1.2 of \cite{WX}]
If $\big(u_0(x),\theta_0(x) \big)\in H^3(\Omega)$ and $u_0(x)$ satisfies the divergence-free condition, then there exists $T_2>0$ such that the inviscid problem \eqref{1.1} admits a unique solution
\begin{equation}\nonumber
\big(u,\theta\big)\in C\big(0,T_2;H^3(\Omega)\big)\cap C^1\big(0,T_2;H^2(\Omega) \big).
\end{equation}
\end{theorem}
The domain considered in \cite{WX} is a boundaed domain with smooth boundary conditions and this case can be naturally extended to periodic domain with periodic boundary conditions.
The proof is due to the argument in \cite{Chae-I} and \cite{LY}. Since it is standard, we omit the details here.
When the dimension $d=2$, Chae and Nam \cite{Chae-N} also proved the local existence and blow-up criterion.

In order to prove the main Theorem, we will need the following Lemma.
\begin{lemma}\label{lemma 3.2}
Let $d=2,3,$ and $k>\frac{d}{2}+1$ be fixed. Let $(u,\theta )\in C(0, T_1; H^k)$ be the corresponding maximal $H^k$-solution of \eqref{1.1} with intial data $(u_0,\theta_0)\in H^k$ and $u_0$ satisfies the divergence-free condition and periodic boundary condition, then $\forall\ 0\leq t< T_1$,
\begin{equation}\label{3.2}
\begin{aligned}
&\|\big( u(t,\cdot),\theta(t,\cdot)\big)\|_{H^k}\leq  \|(u_0,\theta_0) \|_{H^k} \\
&\qquad\times\exp {\left( C_0 \int_0^t \big(1+ \|\nabla u(s,\cdot)\|_{L^\infty}+\|\nabla\theta(s,\cdot)\|_{L^\infty} \big) ds \right)},
\end{aligned}
\end{equation}
where the constant $C_0$ depending on the dimension and $k$.
\end{lemma}
\begin{proof}
Since the initial data $(u_0,\theta_0)\in H^k$ satisfies $\nabla\cdot u_0=0$ and the periodic boundary conditions, the local existence of the $H^k$-solution $(u,\theta)$ is already known. We here only need to show the $H^k$ energy estimate \eqref{3.2}.

Let $\alpha\in\mathbb{N}^d$ satisfies $0\leq |\alpha|=\alpha_1+\alpha_2+\ldots+\alpha_d\leq k$. We first apply the $\partial^\alpha$ on both sides of the first equation of \eqref{1.1} and then take the $L^2$-inner product with $\partial^\alpha u$ with both sides, which gives
\begin{equation}\label{3.3}
\frac{1}{2}\frac{d}{dt}\|\partial^\alpha u\|_{L^2}^2+\langle\partial^\alpha(u\cdot\nabla u),\partial^\alpha u\rangle=\langle\partial^\alpha(\theta e_d),\partial^\alpha u\rangle,
\end{equation}
where the pressure term $\langle\nabla \partial^\alpha p,\partial u\rangle=0$ is due to the fact $u$ is divergence free and the domain considered here is a periodic domain.

We then apply the $\partial^\alpha$ on both sides of the second equation of \eqref{1.1} and take the $L^2$ inner product with $\partial^\alpha \theta$ on both sides, which gives
\begin{equation}\label{3.4}
\frac{1}{2}\frac{d}{dt}\|\partial^\alpha \theta\|_{L^2}^2+\langle\partial^\alpha(u\cdot\nabla \theta),\partial^\alpha\theta\rangle=0.
\end{equation}

Now adding \eqref{3.3} with \eqref{3.4} and taking summation over $0\leq|\alpha|\leq k$, we can obtain
\begin{equation}\label{3.5}
\begin{aligned}
\frac{1}{2}\frac{d}{dt}\|(u,\theta) \|_{H^k}^2 = &\sum_{0\leq |\alpha|\leq k}\langle\partial^\alpha(\theta e_d),\partial^\alpha u\rangle\\
&-\sum_{0\leq |\alpha|\leq k}\langle\partial^\alpha(u\cdot\nabla u),\partial^\alpha u\rangle-\sum_{0\leq |\alpha|\leq k}\langle\partial^\alpha(u\cdot\nabla \theta),\partial^\alpha \theta\rangle.
\end{aligned}
\end{equation}
The H\"older inequality implies that
\begin{equation}\label{3.6}
\big|\sum_{0\leq |\alpha|\leq k}\langle\partial^\alpha(\theta e_d),\partial^\alpha u\rangle\big|\leq \|\theta\|_{H^k}\|u\|_{H^k}.
\end{equation}
Notice that $u$ is divergence free, by use of the Sobolev inequality which can be found in \cite{MB} we can obtain
\begin{equation}\label{3.7}
\big| \sum_{0\leq |\alpha|\leq k}\langle\partial^\alpha(u\cdot\nabla u),\partial^\alpha u\rangle \big| \leq C\|\nabla u\|_{L^\infty}\|u\|_{H^k}^2,
\end{equation}
where the constant $C$ depends on $k$ and the dimension $d$. In the same way, we can obtain
\begin{equation}\label{3.8}
\big| \sum_{0\leq |\alpha|\leq k}\langle\partial^\alpha(u\cdot\nabla \theta),\partial^\alpha \theta\rangle \big| \leq C\big( \|\nabla u\|_{L^\infty}\|u\|_{H^k}+\|\nabla\theta\|_{L^\infty}\|u\|_{H^k}\big)\|\theta\|_{H^k},
\end{equation}
where the constant $C$ also depends on $k$ and the space dimension $d$.
Substituting \eqref{3.6}, \eqref{3.7} and \eqref{3.8} into \eqref{3.5}, we obtain
\begin{equation}\label{3.9}
\begin{aligned}
\frac{1}{2}\frac{d}{dt}\|(u,\theta) \|_{H^k}^2 &\leq C\big(\|\nabla u\|_{L^\infty}+\|\nabla \theta\|_{L^\infty} \big)\|(u,\theta)\|_{H^k}^2+C\|(u,\theta)\|_{H^k}^2\\
&\leq C\big(1+\|\nabla u\|_{L^\infty}+\|\nabla \theta\|_{L^\infty} \big)\|(u,\theta)\|_{H^k}^2,
\end{aligned}
\end{equation}
where $C$ is some constant depending on $k,d$. By the Gronwall inequality, \eqref{3.2} is then proved.
\end{proof}

Lemma \ref{lemma 3.2} tells us that the solution for the inviscid Boussinesq equation has the same Sobolev regularity as the initial data. In the following Lemma we will show that if the initial data $(u_0,\theta_0)\in H^k$ for arbitrary $k\geq3$, there exists an interval $[0,T]$ uniformly with respect to $k$ such that the unique smooth solution $(u,\theta)\in L^\infty([0,T];H^k)$.

\begin{lemma}\label{lemma 3.3}
Let $(u,\theta)$ be the $H^3$-solution of the inviscid Boussinesq equation \eqref{1.1} on the time interval $[0,T]$, with initial data $(u_0,\theta_0)\in H^3$. Then for all $k\geq3$, if $(u_0,\theta_0)\in H^k$, the corresponding solution $(u,\theta)$ satisfies
\begin{equation}\nonumber
(u,\theta)\in L^\infty\big([0,T],H^k \big)\quad \text{for all}\quad k\geq3.
\end{equation}
\end{lemma}

\begin{proof}
We claim that for every $3\leq m\leq k$ there exist a constant $C_m$ such that
\begin{equation}\nonumber
\sup_{0\leq s\leq T}\|\big(u(\cdot,s),\theta(\cdot,s)\big)\|_{H^m}\leq C_m.
\end{equation}
For $m=3$ the statement follows from our assumption. Take $4\leq m\leq k$ and suppose that the statement is true for $m-1$, i. e.
\begin{equation}\nonumber
\sup_{0\leq s\leq T}\|\big(u(\cdot,s),\theta(\cdot,s)\big)\|_{H^{m-1}}\leq C_{m-1}.
\end{equation}
We take the $H^m$ inner product of the first equation of \eqref{1.1} with $u$ and take the $H^m$ inner product of  the second equation of \eqref{1.1} with $\theta$, which gives
\begin{equation}\nonumber
\begin{aligned}
\frac{1}{2}\frac{d}{dt}\|(u,\theta) \|_{H^m}^2 = &\sum_{0\leq |\alpha|\leq m}\langle\partial^\alpha(\theta e_d),\partial^\alpha u\rangle\\
&-\sum_{0\leq |\alpha|\leq m}\langle\partial^\alpha(u\cdot\nabla u),\partial^\alpha u\rangle-\sum_{0\leq |\alpha|\leq m}\langle\partial^\alpha(u\cdot\nabla \theta),\partial^\alpha \theta\rangle.
\end{aligned}
\end{equation}
For $4\leq m\leq k$, by \eqref{3.9} we obtain
\begin{equation}\nonumber
\frac{d}{dt}\|(u,\theta) \|_{H^m} \leq C\big(1+\|\nabla u\|_{L^\infty}+\|\nabla \theta\|_{L^\infty} \big)\|(u,\theta)\|_{H^m},
\end{equation}
where $C$ is a constant depending on $m,d$. Then the Gronwall inequality yields
\begin{equation}\label{3.10}
\begin{aligned}
&\|(u(\cdot,t),\theta(\cdot,t))\|_{H^m}\leq \|(u_0,\theta_0)\|_{H^m}\\
&\qquad\times\exp\bigg(C\int_0^t \big(1+\|\nabla u(\cdot,s)\|_{L^\infty}+\|\nabla \theta(\cdot,s)\|_{L^\infty} \big)ds \bigg).
\end{aligned}
\end{equation}
By Sobolev embedding inequality, we have
\begin{equation}\nonumber
\|\nabla u\|_{L^\infty}\leq C^\prime \|u\|_{H^3},\ \|\nabla \theta\|_{L^\infty}\leq C^\prime \|\theta\|_{H^3},
 \end{equation}
for some constant $C^\prime$. Then from \eqref{3.10} and the assumption, we have
\begin{equation}\nonumber
\begin{aligned}
\|(u(\cdot,t),\theta(\cdot,t))\|_{H^m} &\leq \|(u_0,\theta_0)\|_{H^m}\\
&\times\exp\bigg(C\int_0^t(1+C^\prime\|u(\cdot,s)\|_{H^3}+C^\prime\|\theta(\cdot,s)\|_{H^m})ds \bigg) \\
&\leq \|(u_0,\theta_0)\|_{H^m} \exp\bigg(C(1+2C^\prime C_3)t \bigg).
\end{aligned}
\end{equation}
So it easily follows that
\begin{equation}\nonumber
\sup_{0\leq s\leq T}\|\big(u(\cdot,s),\theta(\cdot,s)\big)\|_{H^m}\leq C_m,
\end{equation}
which proves the Lemma by induction.
\end{proof}
\begin{remark}
In Lemma \ref{lemma 3.3}, the uniform lifespan $[0,T]$ of the Sobolev solution is independent of the Sobolev order $k$ which allows us to take $k\to\infty$. In other words, if the initial datum $(u_0,\theta_0)$ is $C^\infty$, then solution $(u(t,x),\theta(t,x))$ is also $C^\infty$ for almost every $t\in [0,T]$.
\end{remark}

\section{Proof of Theorem \ref{Theorem 2.1}}\label{Section 4}
In this Section, we will give the proof of the main theorem.
\begin{proof}[Proof of Theorem \ref{Theorem 2.1}]
By Lemma \ref{lemma 3.3}, we know that the solution $(u,\theta)$ is smooth, since $(u_0,\theta_0)$ is. Now we claim that for all $|\alpha|=N>2$ we have
\begin{equation}\label{4.1}
\begin{aligned}
&\frac{\|(\partial^\alpha u(t,\cdot),\partial^\alpha \theta(t,\cdot))\|_{H^k}}{M_{|\alpha|}} \leq  2BA^{N-1}(1+C_1Bt)^{N-2} \\
&\qquad \times\exp\bigg(C_0(N-1)\int_0^t(1+\|\nabla u(s)\|_{L^\infty}+\|\nabla \theta(s)\|_{L^\infty})ds \bigg),
\end{aligned}
\end{equation}
where $C_0,C_1$ are positive constants depending only on $k$ and $d$.
We set
\begin{equation}\nonumber
\mathcal{E}_N[(u,\theta)(t)]=\sup_{|\alpha|=N} \frac{\|\partial^\alpha(u(t,\cdot),\theta(t,\cdot))\|_{H^k}}{M_{|\alpha|}}.
\end{equation}
To prove the claim, we proceed by induction on $N$. The result is true for $N=2$ by \eqref{3.2} with notice that $k+2<2k+1$ and $B\geq \frac{9}{4}\|(u_0,\theta_0)\|_{H^{2k+1}},A\geq1$. Hence, let $N\geq3$ and assume \eqref{4.1} holds for multi-indices $\alpha$ of length $2\leq |\alpha|\leq N-1$ and prove it for $|\alpha|=N$.

For $|\alpha|=N, |\gamma|\leq k$, we first apply $\partial^{\alpha+\gamma}$ on both sides of the first and the second equation of \eqref{1.1} and then take the $L^2$-inner product with $\partial^{\alpha+\gamma}u$ and $\partial^{\alpha+\gamma}\theta$ respectivily, which gives
\begin{equation}\label{4.2}
\begin{aligned}
\frac{1}{2}\frac{d}{dt}\|\partial^{\alpha+\gamma}(u,\theta)\|_{L^2}^2 +\langle\partial^{\alpha+\gamma}(u\cdot\nabla u),\partial^{\alpha+\gamma}u\rangle &+\langle\partial^{\alpha+\gamma}(u\cdot\nabla\theta),\partial^{\alpha+\gamma}\theta\rangle\\
&=\langle\partial^{\alpha+\gamma}(\theta e_d),\partial^{\alpha+\gamma}u\rangle.
\end{aligned}
\end{equation}
Denote $Lw=L_uw:=u\cdot\nabla w$ and
\begin{equation}\nonumber
\begin{aligned}
&\mathcal{I}_1=[\partial^{\alpha+\gamma},L]u:=\partial^{\alpha+\gamma}(u\cdot\nabla u)-u\cdot\nabla\partial^{\alpha+\gamma}u,\\
&\mathcal{I}_2=[\partial^{\alpha+\gamma},L]\theta:=\partial^{\alpha+\gamma}(u\cdot\nabla \theta)-u\cdot\nabla\partial^{\alpha+\gamma}\theta,\\
&\mathcal{I}_3=(\partial^{\alpha+\gamma}\theta) e_d.
\end{aligned}
\end{equation}
Taking summation with $0\leq|\gamma|\leq k$ in \eqref{4.2}, we have by the Cauchy-Schwartz inequality
\begin{equation}\label{4.4}
\frac{d}{dt}\|\partial^{\alpha}\big(u(t,\cdot),\theta(t,\cdot) \big)\|_{H^k}\leq \sum_{0\leq|\gamma|\leq k}\|\mathcal{I}_1\|_{L^2}+\sum_{0\leq|\gamma|\leq k}\|\mathcal{I}_2\|_{L^2}+\|\partial^\alpha(u,\theta)\|_{H^k},
\end{equation}
where we used the the standard argument in [pp. 47-48, \cite{R}].
It remains to estimate $\mathcal{I}_1$ and $\mathcal{I}_2$. Note that by the Leibniz rule, we can expand the expression of $\mathcal{I}_1$ as follows
\begin{equation}\nonumber
\begin{aligned}
\mathcal{I}_1 &=\sum_{\beta\leq \alpha}{\alpha\choose\beta}\partial^\gamma(\partial^{\alpha-\beta}u\cdot\nabla\partial^\beta u)\\
                     &=\sum_{\beta\leq \alpha}{\alpha\choose\beta}\sum_{\substack{\delta\leq\gamma,\\ |\beta|+|\delta|<|\alpha|+|\gamma|}}{\gamma\choose\delta}\partial^{\alpha-\beta+\gamma-\delta}u\cdot\nabla\partial^{\beta+\delta}u,
\end{aligned}
\end{equation}
where the restriction $|\beta|+|\delta|<|\alpha|+|\gamma|$ is due to the fact that $\langle u\cdot\nabla \partial^{\alpha+\gamma}u,\partial^{\alpha+\gamma}u \rangle=0$ because $u$ is divergence free.
By \cite{CN}, we have
\begin{equation}\nonumber
\begin{aligned}
\|\mathcal{I}_1\|_{L^2} &\leq \sum_{\beta\leq \alpha}{\alpha\choose\beta}\sum_{\substack{\delta\leq\gamma,\\ |\beta|+|\delta|<|\alpha|+|\gamma|}}{\gamma\choose\delta}\norm{\partial^{\alpha-\beta+\gamma-\delta}u\cdot\nabla\partial^{\beta+\delta}u}_{L^2}\\
                       &\leq C\bigg\{ NM_N\|\nabla u\|_{L^\infty}\mathcal{E}_N[(u,\theta)] +NM_NB^2A^{N-1}\\
                       &\times\exp\bigg[C_0(N-1)\int_0^t (1+\norm{\nabla u(s)}_{L^\infty}+\norm{\nabla \theta(s)}_{L^\infty})ds\bigg](1+C_1Bt)^{N-3} \bigg\}.
\end{aligned}
\end{equation}
We then follow the ideal of \cite{CN} to estimate $\mathcal{I}_2$. Similarly, we can expand $\mathcal{I}_2$ as
\begin{equation}\label{4.6}
\begin{aligned}
\mathcal{I}_2 &=\sum_{\beta\leq \alpha}{\alpha\choose\beta}\partial^\gamma(\partial^{\alpha-\beta}u\cdot\nabla\partial^\beta \theta)\\
                     &=\sum_{\beta\leq \alpha}{\alpha\choose\beta}\sum_{\substack{\delta\leq\gamma,\\ |\beta|+|\delta|<|\alpha|+|\gamma|}}{\gamma\choose\delta}\partial^{\alpha-\beta+\gamma-\delta}u\cdot\nabla\partial^{\beta+\delta}\theta.
\end{aligned}
\end{equation}
Then we divide the summation of the right of \eqref{4.6} into three parts
\begin{equation}\nonumber
\mathcal{I}_2=\mathcal{I}_{21}+\mathcal{I}_{22}+\mathcal{I}_{23},
\end{equation}
where
\begin{equation}\nonumber
\begin{aligned}
\mathcal{I}_{21} =\sum_{\substack{\beta\leq \alpha\\ 0\neq |\beta|\leq |\alpha|-2}}{\alpha\choose\beta}\sum_{\substack{\delta\leq\gamma,\\ |\beta|+|\delta|<|\alpha|+|\gamma|}}{\gamma\choose\delta}\partial^{\alpha-\beta+\gamma-\delta}u\cdot\nabla\partial^{\beta+\delta}\theta,
\end{aligned}
\end{equation}
\begin{equation}\nonumber
\begin{aligned}
\mathcal{I}_{22} = &\sum_{\beta= \alpha}{\alpha\choose\alpha}\sum_{|\delta|=|\gamma|-1}{\gamma\choose\delta}\partial^{\gamma-\delta}u\cdot\nabla\partial^{\alpha+\delta}\theta\\
 &+\sum_{|\beta|=|\alpha|-1}{\alpha\choose\beta}\sum_{\delta=\gamma}{\gamma\choose\gamma}\partial^{\alpha-\beta}u\cdot\nabla\partial^{\beta+\gamma}\theta\\
 &+\sum_{\beta=0}{\alpha\choose\beta}\sum_{\delta=0}{\gamma\choose\delta}\partial^{\alpha+\gamma}u\cdot\nabla \theta,
\end{aligned}
\end{equation}
and
\begin{equation}\nonumber
\begin{aligned}
\mathcal{I}_{23} = &\sum_{\beta= \alpha}{\alpha\choose\alpha}\sum_{|\delta|\leq|\gamma|-2}{\gamma\choose\delta}\partial^{\gamma-\delta}u\cdot\nabla\partial^{\alpha+\delta}\theta\\
 &+\sum_{|\beta|=|\alpha|-1}{\alpha\choose\beta}\sum_{|\delta|\leq |\gamma|-1}{\gamma\choose\gamma}\partial^{\alpha-\beta}u\cdot\nabla\partial^{\beta+\gamma}\theta\\
 &+\sum_{\beta=0}{\alpha\choose\beta}\sum_{\delta\neq0}{\gamma\choose\delta}\partial^{\alpha+\gamma}u\cdot\nabla \theta.
\end{aligned}
\end{equation}
Estimation of $\mathcal{I}_{21}$:
With the fact that $H^k$ is an algbra if $k>\frac{d}{2}+1$ and $|\gamma|\leq k$, we have
\begin{equation}\label{4.10}
\begin{aligned}
\|\mathcal{I}_{21}\|_{L^2} &\leq \sum_{\substack{\beta\leq \alpha\\ 0\neq |\beta|\leq |\alpha|-2}}{\alpha\choose\beta}\sum_{\substack{\delta\leq\gamma,\\ |\beta|+|\delta|<|\alpha|+|\gamma|}}{\gamma\choose\delta}\norm{\partial^{\alpha-\beta+\gamma-\delta}u\cdot\nabla\partial^{\beta+\delta}\theta}_{L^2}\\
                               &\leq C\sum_{\substack{\beta\leq \alpha\\ 0\neq |\beta|\leq |\alpha|-2}}{\alpha\choose\beta}\sum_{\substack{\delta\leq\gamma,\\ |\beta|+|\delta|<|\alpha|+|\gamma|}}{\gamma\choose\delta}\norm{\partial^{\alpha-\beta}u}_{H^k}\norm{\nabla\partial^{\beta}\theta}_{H^k}.
\end{aligned}
\end{equation}
Noting that $2\leq |\alpha-\beta|\leq N-1$ and $2\leq |\beta|+1\leq N-1$, the hypothesis \eqref{4.1} for $2\leq |\alpha|\leq N-1$ indicates that
\begin{equation}\label{4.11}
\begin{aligned}
 &\norm{\partial^{\alpha-\beta}u}_{H^k}\leq  M_{|\alpha-\beta|}2BA^{|\alpha-\beta|-1} \\
 &\quad\times\exp\big[C_0(|\alpha-\beta|-1)\int_0^t (1+\norm{\nabla u(s)}_{L^\infty}+\norm{\nabla \theta(s)}_{L^\infty})ds \big](1+C_1Bt)^{|\alpha-\beta|-2}
\end{aligned}
\end{equation}
and
\begin{equation}\label{4.12}
\begin{aligned}
&\norm{\nabla\partial^\beta\theta}_{H^k}\leq M_{|\beta|+1}2BA^{|\beta|}\\
&\quad\times \exp\big[C_0(|\beta|)\int_0^t (1+\norm{\nabla u(s)}_{L^\infty}+\norm{\nabla \theta(s)}_{L^\infty})ds \big](1+C_1Bt)^{|\beta|-1}.
\end{aligned}
\end{equation}
Substituting \eqref{4.11} and \eqref{4.12} into \eqref{4.10} and employing \eqref{2.1+}, we obtain
\begin{equation}\nonumber
\begin{aligned}
\|\mathcal{I}_{21}\|_{L^2} \leq & C\sum_{\substack{\beta\leq \alpha\\ 0\neq |\beta|\leq |\alpha|-2}}{\alpha\choose\beta}M_{|\alpha-\beta|}M_{|\beta|+1}B^2 A^{N-1} \\
&\quad\times \exp \big[C_0(N-1)\int_0^t (1+\norm{\nabla u(s)}_{L^\infty}+\norm{\nabla \theta(s)}_{L^\infty})ds \big](1+C_1Bt)^{N-3}\\
 \leq & CNM_{N}B^2 A^{N-1} \\
&\quad \times\exp \big[C_0(N-1)\int_0^t (1+\norm{\nabla u(s)}_{L^\infty}+\norm{\nabla \theta(s)}_{L^\infty})ds \big](1+C_1Bt)^{N-3}.
\end{aligned}
\end{equation}
Estimation of $\mathcal{I}_{22}$:
In a similar way, we rewrite $\mathcal{I}_{22}$ as
\begin{equation}\nonumber
\begin{aligned}
\mathcal{I}_{22} &=\sum_{\beta= \alpha}{\alpha\choose\alpha}\sum_{|\delta|=|\gamma|-1}{\gamma\choose\delta}\partial^{\gamma-\delta}u\cdot\nabla\partial^{\alpha+\delta}\theta\\
 &\quad+\sum_{|\beta|=|\alpha|-1}{\alpha\choose\beta}\sum_{\delta=\gamma}{\gamma\choose\gamma}\partial^{\alpha-\beta}u\cdot\nabla\partial^{\beta+\gamma}\theta\\
 &\quad+\sum_{\beta=0}{\alpha\choose\beta}\sum_{\delta=0}{\gamma\choose\delta}\partial^{\alpha+\gamma}u\cdot\nabla \theta\\
 &=\mathcal{R}_{21}+\mathcal{R}_{23}+\mathcal{R}_{23}.
\end{aligned}
\end{equation}
For $\mathcal{R}_{21}$, we obtain
\begin{equation}\label{4.16}
\begin{aligned}
\|\mathcal{R}_{21}\|_{L^2} &\leq \sum_{\beta= \alpha}{\alpha\choose\alpha}\sum_{|\delta|=|\gamma|-1}{\gamma\choose\delta}\|{\partial^{\gamma-\delta}u\cdot\nabla\partial^{\alpha+\delta}\theta}\|_{L^2}\\
                                  &\leq \sum_{\beta= \alpha}{\alpha\choose\alpha}\sum_{|\delta|=|\gamma|-1}{\gamma\choose\delta}\norm{\partial^{\gamma-\delta}u\cdot\nabla\partial^{\alpha+\delta}\theta}_{L^2}\\
                                 &\leq C\norm{\nabla u}_{L^\infty}\norm{\partial^\alpha\theta}_{H^k}.
\end{aligned}
\end{equation}
for some constant $C$ depending on $k$.

For $\mathcal{R}_{22}$, we have the following estimate
\begin{equation}\label{4.17}
\begin{aligned}
\mathcal{R}_{22} &\leq \sum_{|\beta|=|\alpha|-1}{\alpha\choose\beta}\sum_{\delta=\gamma}{\gamma\choose\gamma}\|{\partial^{\alpha-\beta}u\cdot\nabla\partial^{\beta+\gamma}\theta}\|_{L^2}\\
                                  &\leq \sum_{|\beta|=|\alpha|-1}N\sum_{\delta=\gamma}{\gamma\choose\gamma}\norm{\partial^{\alpha-\beta}u\cdot\nabla\partial^{\beta+\gamma}\theta}_{L^2}\\
 &\leq CdNM_N\norm{\nabla u}_{L^\infty}\mathcal{E}_{N}[\theta].
\end{aligned}
\end{equation}
Then for $\mathcal{R}_{23}$, we have
\begin{equation}\label{4.18}
\begin{aligned}
\|\mathcal{R}_{23}\|_{L^2} &\leq\sum_{\beta=0}{\alpha\choose\beta}\sum_{\delta=0}{\gamma\choose\delta}\|{\partial^{\alpha+\gamma}u\cdot\nabla \theta}\|_{L^2}\\
                                  &\leq \norm{\nabla \theta}_{L^\infty}\norm{\partial^\alpha u}_{H^k}.
\end{aligned}
\end{equation}
Summing up the estimates of \eqref{4.16}, \eqref{4.17} and \eqref{4.18}, we obtain
\begin{equation}\nonumber
\|\mathcal{I}_{22}\|_{L^2} \leq C\big( \|\nabla u\|_{L^\infty}\|\partial^\alpha\theta\|_{H^k}+\|\nabla \theta\|_{L^\infty}\|\partial^\alpha u\|_{H^k}+dNM_N\|\nabla u\|_{L^\infty}\mathcal{E}_N(\theta) \big).
\end{equation}
Estimate of $\mathcal{I}_{23}$:
We divide $\mathcal{I}_{23}$ as follows
\begin{equation}\nonumber
\begin{aligned}
\mathcal{I}_{23} &= \sum_{\beta= \alpha}{\alpha\choose\alpha}\sum_{|\delta|\leq|\gamma|-2}{\gamma\choose\delta}\partial^{\gamma-\delta}u\cdot\nabla\partial^{\alpha+\delta}\theta\\
                                &\quad+\sum_{|\beta|=|\alpha|-1}{\alpha\choose\beta}\sum_{|\delta|\leq |\gamma|-1}{\gamma\choose\gamma}\partial^{\alpha-\beta+\gamma-\delta}u\cdot\nabla\partial^{\beta+\gamma}\theta\\
                                &\quad+\sum_{\beta=0}{\alpha\choose\beta}\sum_{\delta\neq0}{\gamma\choose\delta}\partial^{\alpha+\gamma-\delta}u\cdot\nabla \partial^\delta\theta\\
                                &=\mathcal{R}_{31}+\mathcal{R}_{32}+\mathcal{R}_{33}.
\end{aligned}
\end{equation}
For $\mathcal{R}_{31}$, we have
\begin{equation}\label{4.21}
\begin{aligned}
&\|\mathcal{R}_{31}\|_{L^2} \leq C\sum_{\beta= \alpha}{\alpha\choose\alpha}\sum_{|\delta|\leq|\gamma|-2}{\gamma\choose\delta}\|\partial^{\gamma-\delta}u\cdot\nabla\partial^{\alpha+\delta}\theta\|_{L^2}\\
                                  &\leq C\|\partial^{\gamma-\delta} u\|_{L^\infty}\|\nabla\partial^{\alpha+\delta}\theta\|_{L^2}\\
                                  &\leq C\|\partial^{\gamma-\delta}u \|_{H^k}\|\partial^\alpha \theta \|_{H^{k-1}}\\
                                  &\leq C\|\partial^{\gamma-\delta}(u,\theta)\|_{H^k}\|\partial^\alpha(u,\theta)\|_{H^{k-1}}\\
                                  &\leq C\|(u_0,\theta_0)\|_{H^{2k}}\exp\big[C_0\int_0^t(1+\|\nabla u\|_{L^\infty}+\|\nabla \theta\|_{L^\infty})ds\big]\sup_{j:\alpha_j\geq1}\|\partial^{\alpha-e_j}(u,\theta)\|_{H^k}\\
                                  &\leq C\|(u_0,\theta_0)\|_{H^{2k}}\exp\big[C_0\int_0^t(1+\|\nabla u\|_{L^\infty}+\|\nabla \theta\|_{L^\infty})ds\big]\\
                                  &\times M_{N-1}BA^{N-2}\exp\big[C_0(N-2)\int_0^t(1+\|\nabla u\|_{L^\infty}+\|\nabla \theta\|_{L^\infty})ds\big](1+C_1Bt)^{N-3}\\
                                  &\leq CM_{N-1}B^2A^{N-2}\exp\big[C_0(N-1)\int_0^t(1+\|\nabla u\|_{L^\infty}+\|\nabla \theta\|_{L^\infty})ds\big](1+C_1Bt)^{N-3},
\end{aligned}
\end{equation}
where we used the fact $|\gamma-\delta|+k\leq 2k$ in \eqref{4.21}, the inductive hypothesis \eqref{4.1}, and the fact that $B>\|(u_0,\theta_0)\|_{H^{2k+1}}$.
For $\mathcal{R}_{32}$, we have
\begin{equation}\label{4.22}
\begin{aligned}
&\|\mathcal{R}_{32}\|_{L^2} \leq \sum_{|\beta|=|\alpha|-1}{\alpha\choose\beta}\sum_{|\delta|\leq |\gamma|-1}{\gamma\choose\gamma}\|{\partial^{\alpha-\beta+\gamma-\delta}u\cdot\nabla\partial^{\beta+\gamma}\theta}\|_{L^2}\\
                               &\leq |\alpha|\|\partial^{\alpha-\beta+\gamma-\delta}u\|_{L^\infty}\|\nabla\partial^{\beta+\delta}\theta\|_{L^2}\\
                               &\leq C|\alpha| \|\partial^{\gamma-\delta}u\|_{H^k}\| \partial^\beta\theta\|_{H^k}\\
                               &\leq C|\alpha| \|\partial^{\gamma-\delta}(u,\theta)\|_{H^k}\| \partial^\beta(u,\theta)\|_{H^k}\\
                               &\leq CNM_{N-1}\|(u_0,\theta_0) \|_{H^{2k}}\exp[C_0\int_0^t(1+\|\nabla u\|_{L^\infty}+\|\nabla\theta\|_{L^\infty})ds]\\
                               &\times BA^{N-2}\exp\big[C_0(N-2)\int_0^t(1+\|\nabla u\|_{L^\infty}+\|\nabla \theta\|_{L^\infty})ds\big](1+C_1Bt)^{N-3}\\
                               &\leq CM_{N-1}B^2A^{N-2}\exp\big[C_0(N-1)\int_0^t(1+\|\nabla u\|_{L^\infty}+\|\nabla \theta\|_{L^\infty})ds\big](1+C_1Bt)^{N-3}.
\end{aligned}
\end{equation}
For $\mathcal{R}_{33}$, in a similar way we have
\begin{equation}\label{4.23}
\begin{aligned}
&\|\mathcal{R}_{33}\|_{L^2} \leq \sum_{\beta=0}{\alpha\choose\beta}\sum_{\delta\neq0}{\gamma\choose\delta}\|{\partial^{\alpha+\gamma-\delta}u\cdot\nabla \partial^\delta\theta}\|_{L^2}\\
                               &\leq C\|\partial^\alpha u\|_{H^{k-1}}\|\nabla\partial^\delta \theta\|_{H^k}\\
                               &\leq CM_{N-1}B^2A^{N-2}\exp\big[C_0(N-1)\int_0^t(1+\|\nabla u\|_{L^\infty}+\|\nabla \theta\|_{L^\infty})ds\big](1+C_1Bt)^{N-3}.
\end{aligned}
\end{equation}
Summing up \eqref{4.21},\eqref{4.22} and \eqref{4.23}, we obtain
\begin{equation}\nonumber
\begin{aligned}
\|\mathcal{I}_2\|_{L^2} &\leq CNM_N(\|\nabla u\|_{L^\infty}+\|\nabla \theta\|_{L^\infty})\mathcal{E}_N[(u,\theta)]\\
                           &\quad+CNM_NB^2A^{N-1}\exp\bigg(C_0(N-1)\int_0^t\big[1+\norm{\nabla u(s)}_{L^\infty}+\norm{\nabla \theta(s)}_{L^\infty}\big]ds\bigg).
\end{aligned}
\end{equation}
Combining the estimates $\mathcal{I}_1$ and $\mathcal{I}_2$, we have from \eqref{4.4}
\begin{equation}\label{4.24}
\begin{aligned}
&\frac{d}{dt}\frac{\|\partial^\alpha(u,\theta)(t)\|_{H^k}}{M_{|\alpha|}}\leq C(N-1)(1+\|\nabla u(t)\|_{L^\infty}+\|\nabla\theta(t)\|_{L^\infty})\mathcal{E}_{N}[(u,\theta)(t)]\\
&\qquad+CNA^{N-1}B^2\exp\bigg(C_0(N-1)\int_0^t(1+\|\nabla u(s)\|_{L^\infty}+\|\nabla \theta(s)\|_{L^\infty})ds \bigg)\\
&\qquad\qquad\times(1+C_1Bt)^{N-3},
\end{aligned}
\end{equation}
where we used the Cauchy-Schwartz inequality and the fact $\frac{N}{N-1}\leq \frac{3}{2}$ if $N\geq3$ and the constant $C>0$ depending only on the dimension $d$ and $k$.

Now we integrate \eqref{4.24} from $0$ to $t$ and take the supremum on $|\alpha|=N$. We obtain
\begin{equation}\label{4.25}
\begin{aligned}
\mathcal{E}_N[(u,\theta)(t)] &\leq \mathcal{E}_N[(u_0,\theta_0)]\\
&\quad+\int_0^t C(N-1)(1+\|\nabla u(t)\|_{L^\infty}+\|\nabla\theta(t)\|_{L^\infty})\mathcal{E}_{N}[(u,\theta)(s)]ds\\
&\quad+\int_0^t CNA^{N-1}B^2 (1+C_1Bs)^{N-3}\\
&\quad\quad\quad \times \exp\bigg(C_0(N-1)\int_0^s(1+\|\nabla u(\ell)\|_{L^\infty}+\|\nabla \theta(\ell)\|_{L^\infty})d\ell \bigg)ds.
\end{aligned}
\end{equation}

We can take $C_0\geq C$ in \eqref{4.25}, so that Grownwall inequality [Lemma 2.1 in \cite{CN}] gives
\begin{equation}\nonumber
\begin{aligned}
\mathcal{E}_N[(u,\theta)(t)] & \leq \exp\bigg(C_0(N-1)\int_0^t (1+\|\nabla u\|_{L^\infty}+\|\nabla\theta\|_{L^\infty})ds \bigg)\\
&\quad\times \big[\mathcal{E}_N[(u,\theta)(0)] +CNB^2A^{N-1}\int_0^t(1+C_1Bs)^{N-3}ds \big]\\
                                                    &\leq  \exp\bigg(C_0(N-1)\int_0^t (1+\|\nabla u\|_{L^\infty}+\|\nabla\theta\|_{L^\infty})ds \bigg) \\
&\quad\times \big[\mathcal{E}_N[(u,\theta)(0)]+\frac{CN}{C_1(N-2)}BA^{N-1}(1+C_1Bt)^{N-2}\big].
\end{aligned}
\end{equation}

Note that we have $\mathcal{E}_N[(u,\theta)(0)]\leq BA^{N-1}$ by the assumption \eqref{2.2}. If we choose $C_1=3C$, so that $\frac{3C}{C_1}=1$, since $A\geq1, N\geq3$ we have
\begin{equation}\nonumber
\begin{aligned}
\mathcal{E}_N[(u,\theta)(0)] &+\frac{CN}{C_1(N-2)}BA^{N-1}(1+C_1Bt)^{N-2}\\
                                                    &\leq \big[BA^{N-1}+\frac{3C}{C_1}BA^{N-1}(1+C_1Bt)^{N-2} \big]\\
                                                    &\leq 2BA^{N-1}(1+C_1Bt)^{N-2},
\end{aligned}
\end{equation}
and we obtain exactly \eqref{4.1} for $|\alpha|=N$. Then the theorem is proved.

\end{proof}

\bigskip
\noindent{\bf Acknowledgements.}
The research of the second author is supported partially by
``The Fundamental Research Funds for Central Universities of China".

\end{document}